\documentclass[12pt]{article}

\usepackage{amsmath,amssymb}
\usepackage[pdftex]{graphicx}

\newtheorem{theorem}{Theorem}[section]
\newtheorem{lemma}[theorem]{Lemma}
\newtheorem{corollary}[theorem]{Corollary}

\newtheorem{definition}[theorem]{Definition}
\newtheorem{example}[theorem]{Example}
\newtheorem{conjecture}[theorem]{Conjecture}
\newtheorem{remark}[theorem]{Remark}

\newenvironment{proof}{\noindent\textsc{Proof: }}
{\hspace{\stretch{1}}$\Box$\medskip}

\begin{document}

\title{A local criterion for Tverberg graphs}

\author{
Alexander Engstr\"om\footnote{The author was a Miller Research Fellow 2009-2011 at UC Berkeley, and gratefully acknowledges 
support from the Adolph C. and Mary Sprague Miller Institute for Basic Research in Science.}\\ 
Department of Mathematics\\
Aalto University\\
Helsinki, Finland\\
  {\tt alexander.engstrom@aalto.fi}}

\date\today

\maketitle

\abstract{
   The topological Tverberg theorem
   states that for any prime power $q$ 
   and continuous map from a
   $(d+1)(q-1)$-simplex to $\mathbb{R}^d$,
   there are $q$ disjoint faces $F_i$
   of the simplex whose images intersect.
   It is possible to put conditions on
   which pairs of vertices of the simplex
   that are allowed to be in the same face $F_i$.
   A graph with the same vertex set as the simplex, 
   and with two vertices adjacent if they should not
   be in the same $F_i$, is called a Tverberg graph
   if the topological Tverberg theorem still work.
   
   These graphs have been studied by Hell, Sch\"oneborn and Ziegler,
   and it is known that disjoint unions of small paths,
   cycles, and complete graphs are Tverberg graphs.
   We find many new examples by establishing a local criterion
   for a graph to be Tverberg. An easily stated corollary of our
   main theorem is that if the maximal 
   degree of a graph is $D$, and \mbox{$D(D+1)<q$}, then it is a
   Tverberg graph.
   
   We state the affine versions of our results and also describe how
   they can be used to enumerate Tverberg partitions.
}

\section{Introduction}

The topological Tverberg theorem is a popular
subject among researchers who work both with
topological and combinatorial methods. It
states that for any prime power $q$ and
continuous map $f$ from a  ${(d+1)(q-1)}$--dimensional 
simplex to $\mathbb{R}^d$, there are $q$ disjoint faces
$F_1,F_2,\ldots,F_q$ such that
$\cap_{i=1}^q f(F_i)$ is non-empty.

Vu\v ci\'c and \v Zivaljevi\'c \cite{VZ},
Sch\"oneborn and Ziegler \cite{SZ}, and 
Hell \mbox{\cite{H1,H2}} 
investigated
if it is possible to put further conditions 
on the faces $F_1,F_2,\ldots,F_q$.
A graph $G$ with the same vertex set as the
 ${(d+1)(q-1)}$--dimensional simplex,
is called a $(d,q)$--\emph{Tverberg graph} if we 
for any continuous map from the simplex
to $\mathbb{R}^d$ can find disjoint faces 
$F_1,F_2,\ldots F_q$ whose images intersect
and moreover vertices adjacent in $G$
are in different $F_i$:s. When the dimension
of the space which the simplex is mapped into
and the number of disjoint faces is clear,
the $(d,q)$ in front of ``Tverberg graph''
is dropped.

Hell \cite{H2} proved that
one can build up a Tverberg graph
as a disjoint union of complete graphs
with less than $q/2+1$ vertices, stars
with less than $q$ vertices, and paths and
cycles when $q>4$. He used connectivity
results on chessboard complexes from \cite{BLVZ}
in the spirit of \cite{VZ2} for the complete
graphs and explicit calculations for the
other classes.

In a graph the set of vertices adjacent to a vertex $v$ is denoted $N(v)$; and the set of vertices on distance two from $v$ is
$N^2(v)=  ( \cup_{u\in N(v)} N(u) ) \setminus ( N(v) \cup \{v\} ).$

The main theorem of the paper gives a local condition on graphs which guarantee that they are Tverberg graphs.
\newline
\newline
{\bf Theorem \ref{mainTheorem}} 
\emph{Let $q\geq 2$ be a prime power, $d\geq 1$, and $G$ a graph on $(d+1)(q-1)+1$ vertices such that  
\[q> |N^2(v)| + 2|N(v)| \]
for every vertex $v$. Then $G$ is a $(d,q)$--Tverberg graph.}
\newline
\newline
\indent
A weaker but easily applied corollary follows from the main theorem.
\newline
\newline
{\bf Corollary \ref{mainCorollary}} 
\emph{Let $q\geq 2$ be a prime power, $d\geq 1$, and $G$ a graph on $(d+1)(q-1)+1$ vertices such that its maximal degree $D$ satisfy $D(D+1)<q$. Then $G$ is a $(d,q)$--Tverberg graph.}
\newline
\newline
\indent
When there are no edges in $G$ this is the
ordinary topological Tverberg theorem.

\begin{example}
\emph{
With the parameters $d=2$ and $q=2^4$, the ordinary
topological Tverberg theorem states that for any continuous map $f$ from a  $45$--dimensional 
simplex to $\mathbb{R}^2$, there are $16$ disjoint faces
$F_1,F_2,\ldots,F_{16}$ of the simplex such that
$\cap_{i=1}^{16} f(F_i)$ is non-empty.}

\begin{figure}
\centering 
\includegraphics[width=8cm]{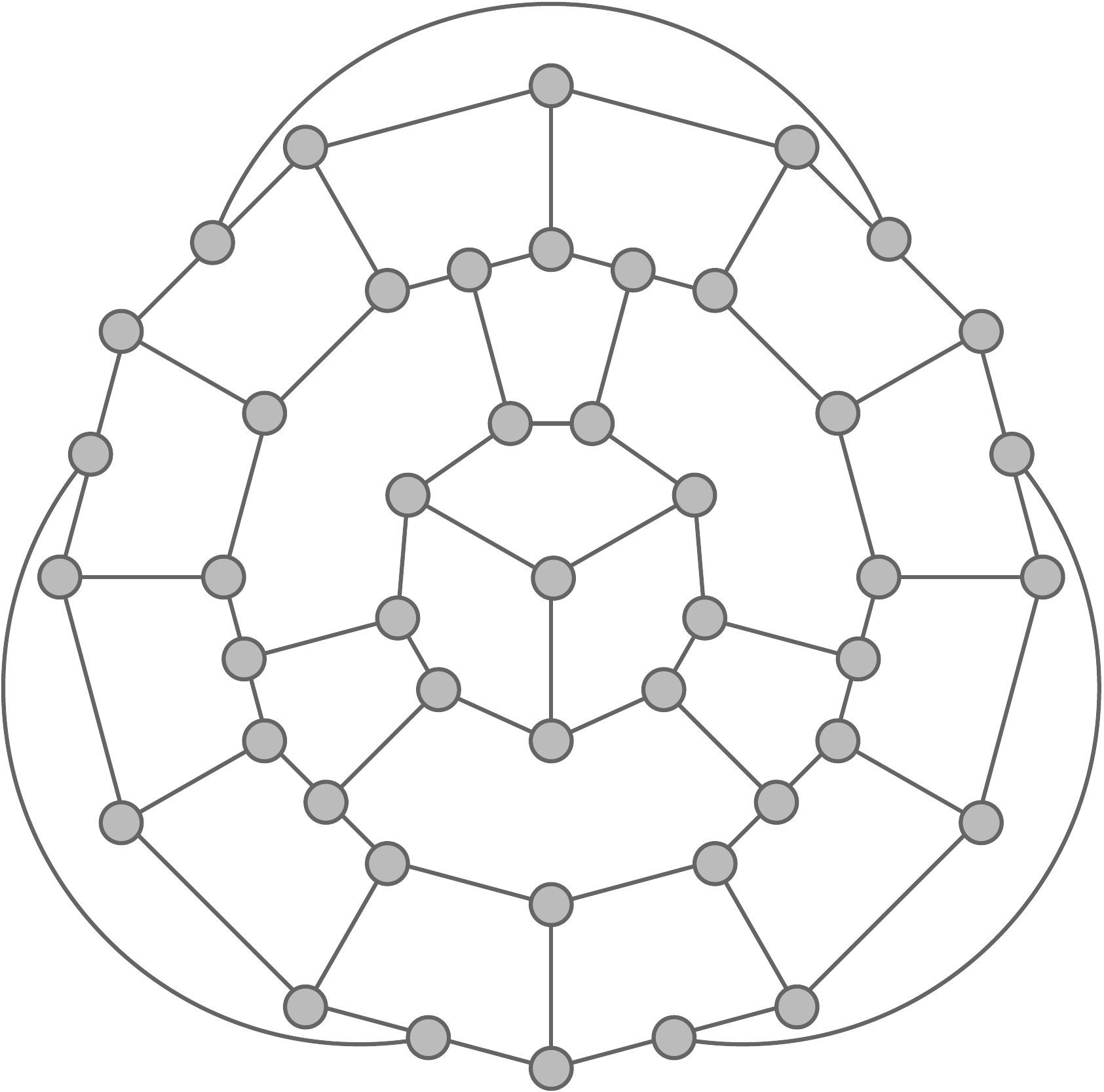}
\caption{The Grinberg graph}  \label{Grinberg}
\end{figure} 

\emph{
The \emph{Grinberg graph} (Figure~18.9 of \cite{BM}) is a 3-regular graph on 46 vertices described in Figure~\ref{Grinberg}. 
Identify the vertices of the graph with those of a $45$--dimensional simplex.
Let $D=3$
be the maximal degree of the Grinberg graph. 
The condition $D(D+1)=3\cdot 4=12<15=2^4-1=q-1$ is satisfied and we can for any continuous map $f$ apply Corollary~\ref{mainCorollary}.
We get 16 disjoint faces $F_1,F_2,\ldots,F_{16}$ of the simplex such that both
\begin{itemize}
\item the intersection $\cap_{i=1}^{16} f(F_i)$ is non-empty, and
\item any two vertices in the same $F_i$ are \emph{not} adjacent in the Grinberg graph.
\end{itemize}}
\end{example}

\begin{remark}
\emph{
After this paper was submitted, Blagojevi\'c, Matschke and Ziegler made significant progress on the colored Tverberg problem \cite{BLZ1,BLZ2}.
They proved that disjoint unions of complete graphs on at most $q-1$ vertices are Tverberg graphs, the optimal result.
For the particular case of disjoint unions of complete graphs, our Theorem~\ref{mainTheorem} only allow complete graphs
with at most $(q-1)/2$ vertices, as proved before in \cite{VZ2}.
}
\end{remark}

\subsection{Notation}

The notation and methods used 
in this paper are described by\linebreak
Bj\"orner~\cite{B} and Matou\v sek~\cite{M}.
Everything what is needed about the independence
complexes used in the next section is
included there, but for the larger
picture see for example 
Dochtermann, Engstr\"om~\cite{DE,E} and their references.

\section{Independence complex formulation}

All graphs are finite and simple. The
vertex set of $G$ is $V(G)$ and the
edge set $E(G)$. The neighborhood $N_G(v)$
of a vertex $v$ of $G$ is the set of
vertices adjacent to $v$ in $G$. Sometimes
the $G$ in $N_G(v)$ is dropped.
The induced subgraph of $G$ on $S\subseteq V(G)$
has vertex set $S$ and all possible edges
of $G$ with those vertices. A subset $S$ of
$V(G)$ is independent if the induced subgraph
of $G$ on $S$ lacks edges.

The \emph{independence complex} ${\tt Ind}(G)$ of a graph
$G$ is the simplicial complex with the same vertex
set as $G$ and its faces are the independent
sets of $G$. The following graph product will be used
many times to produce new graphs.

\begin{definition}
The cartesian product $G \square H$ of two graphs $G$ and $H$ is a graph with vertex set $V(G) \times V(H)$ and edge set
\[ \{ \{(v_1,w),(v_2,w)\} \mid \{v_1,v_2\} \in E(G), w\in V(H) \} \]
union
\[ \{ \{(v,w_1),(v,w_2)\} \mid  v\in V(G),  \{w_1,w_2\} \in E(H) \}. \]
\end{definition}

Another way to introduce $G \square H$, is as the 1-skeleton of the polyhedral product of $G$ and $H$ seen as cell complexes.
A basic but sometimes useful fact is that $G \square H$Êand $H \square G$ are isomorphic.
For us, one of the graphs in the product will almost always be a complete one. In Figure~\ref{produktFigur} is
the product of a path on six vertices and a complete graph on four vertices.

\begin{figure}
\centering 
\includegraphics[width=8cm]{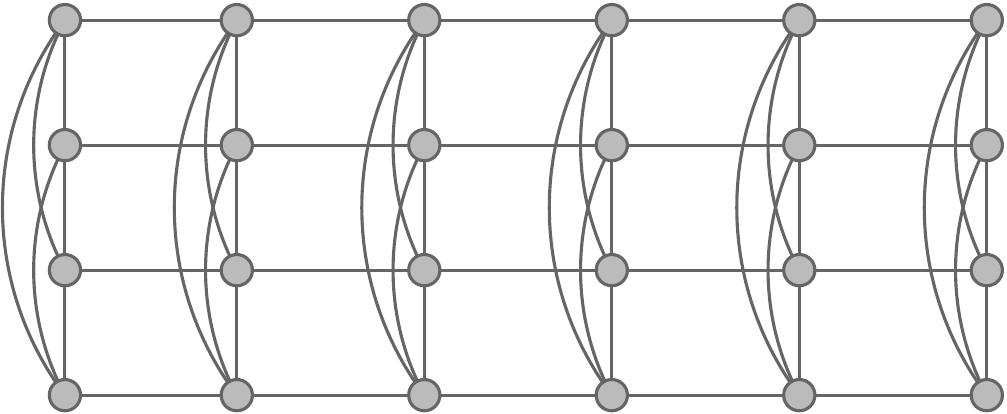}
\caption{The product of a path on six vertices and a complete graph on four vertices. Each copy of the path constitutes a level of the product.}  \label{produktFigur}
\end{figure} 

If $G$ is a graph with vertex set $\{0,1,2,\ldots,N\}$
and no edges, then the distinct vertices
$(u,i)$ and $(v,j)$ of $G\square K_q$ can only be in
the same face $\sigma\in {\tt Ind}(G\square K_q)$
if $u\neq v$. Another way to state it, is
that ${\tt Ind}(G\square K_q)$ is the $q$-fold
2-wise deleted join of the simplex
$\sigma^N$ with vertex set $\{0,1,2,\ldots,N\}$.
In short, 
${\tt Ind}(G\square K_q)=(\sigma^N)^{\ast q}_{\Delta(2)}$.
This is important because $(\sigma^N)^{\ast q}_{\Delta(2)}$
is central in the textbook proof \cite{M} of the
topological Tverberg theorem.

An important property of ${\tt Ind}(G\square K_q)$ for
$q=p^r$, where $p$ is a prime, is that the
 action of $\mathbb{Z}_p\times \cdots
\times \mathbb{Z}_p$ on $V(G\square K_q)$ by permuting
the second component of the vertices,
induces a fixed-point free action on 
${\tt Ind}(G\square K_q)$. 

A simplicial complex $\Delta$ is $n$--\emph{connected}
if for all $0\leq i\leq n$ any continuous map from
the $i$--sphere to $\Delta$ can be extended to
a continuous map from the $(i+1)$--ball to $\Delta$.
Note the degenerate case that $\Delta$ is $(-1)$--connected
if it is non-empty.

Hell \cite{H1,H2}, modified B\'ar\'any, Shlosman,
Sz\H ucs \cite {BSS}, Sarkaria \cite{S}
 and Volovikov's arguments \cite{V} to prove:

\begin{theorem}\label{modTver}
  Let $q\geq 2$ be a prime power, $d\geq 1$, and
  $N=(d+1)(q-1)$. If 
  $G$ is a graph with $N+1$ vertices
  such that ${\tt Ind}(G\square K_q)$ is $(N-1)$--connected
  then $G$ is a $(d,q)$--Tverberg graph.
\end{theorem}

However, it was not formulated with independence complexes.
Using the proof of the topological Tverberg theorem in
Matou\v sek's book (Theorem 6.4.2, page 162, \cite{M}) as a blue-print, and
substituting $(\sigma^N)^{\ast q}_{\Delta(2)}$ by
${\tt Ind}(G\square K_q)$, one gets a proof of Theorem~\ref{modTver}
in the $q$ prime case. Matou\v sek also describes (in the notes, page 164--165, \cite{M}) the 
straightforward generalization of the topological Tverberg theorem
to the $q$ prime power case and the full proof of
Theorem~\ref{modTver} follows with similar modifications.

\section{Connectivity of ${\tt Ind}(G\square K_q)$ and squids.}

In this section we first prove two reduction
lemmas for independence complexes and then use
them to calculate the connectivity of\linebreak ${\tt Ind}(G\square K_q)$.
The reduction procedure is done by removing
certain sets of vertices called squids. 
Finally we give the proof of Theorem~\ref{mainTheorem}
and state some conjectures.

\subsection{Connectivity calculations}

\begin{lemma}\label{red1}
  Let $v$ be a vertex of $H$.
  If ${\tt Ind}(H\setminus \{v\})$ is $n$--connected
  and ${\tt Ind}(H\setminus (N(v)\cup \{v\})  )$ is $(n-1)$--connected
  then ${\tt Ind}(H)$ is $n$--connected.
\end{lemma}

\begin{proof}
  We will use a simple gluing result (Lemma 10.3 (ii), \cite{B}): If
  $\Delta_1$ and $\Delta_2$ are $n$--connected and $\Delta_1\cap
  \Delta_2$ is $(n-1)$--connected, then $\Delta_1\cup\Delta_2$ 
  is $n$--connected.

  The complex $\Delta_1=\{ \sigma\in {\tt Ind}(H) \mid
       \sigma\cup\{v\}\in {\tt Ind}(H) \}$ 
  is $n$--connected since it is a cone with apex $v$.
  The complex $\Delta_2=\{ \sigma\in {\tt Ind}(H) \mid
    v\not\in \sigma\}= {\tt Ind}(H\setminus \{v\}) $
   is $n$--connected by assumption, and
   $\Delta_1\cap \Delta_2= {\tt Ind}(H\setminus (N(v)\cup \{v\})  )$
   is $(n-1)$--connected by assumption.
   Hence ${\tt Ind}(H) = \Delta_1\cup \Delta_2$
   is $n$--connected.
\end{proof}

\begin{lemma}\label{red2}
   Let $v$ be a vertex of $H$ whose
   neighborhood is a complete graph. If
   ${\tt Ind}(H\setminus (N(v)\cup N(u)))$ is
   $(n-1)$--connected for all $u\in N(v)$ then
   ${\tt Ind}(H)$ is $n$--connected.
\end{lemma}
\begin{proof}
   This is Lemma 10.4 (ii) of \cite{B}: If
   $\Delta=\Delta_0\cup\Delta_1\cup\cdots\cup\Delta_m$
   is a simplicial complex with contractible
   subcomplexes $\Delta_i$ for all $0\leq i \leq m$,
   and  $\Delta_0 \supseteq \Delta_i\cap \Delta_j$
   for all $1\leq i<j\leq m$, then $\Delta$
   is homotopy equivalent with
   $$ \bigvee_{1\leq i\leq m} \textrm{susp}(\Delta_0 \cap
   \Delta_i). $$
   We will use a connectivity version of the lemma.
   If all $\Delta_0 \cap \Delta_i$ are \mbox{$(n-1)$}--connected,
   then their suspensions are $n$--connected, the
   wedge of the suspensions is $n$--connected, and by
   homotopy equivalence so is $\Delta$.

   Let $\Delta_0= {\tt Ind}(H\setminus (N(v)))$ and
   $\Delta_u={\tt Ind}(H\setminus (N(u)))$ for
   $u\in N(v)$. All of these complexes are
   contractible since they are cones. 

   If
   $\sigma\in {\tt Ind}(H)$ and one can add $v$ to
   $\sigma$ and stay in ${\tt Ind}(H)$, then
   $\sigma\in \Delta_0$. If one cannot add $v$ to
   $\sigma$ and stay in ${\tt Ind}(H)$, then there
   is a vertex $u\in N(v)$ which is in $\sigma$,
   and $\sigma\in \Delta_u$. Thus
   ${\tt Ind}(H)=\Delta_0\cup\bigcup_{u\in N(v)} \Delta_u$.
   If $u$ and $w$ are different vertices in
   $N(v)$ and $\sigma \in \Delta_u\cap\Delta_w$ then
   $\sigma\cap N(v) = \emptyset$ and we can add
   $v$ to $\sigma$ without falling out of ${\tt Ind}(H)$.
   This shows that $\Delta_0\supseteq \Delta_u\cap \Delta_w$.
   By assumption $\Delta_0\cap\Delta_u={\tt Ind}(H\setminus (N(v)\cup N(u)))$ is
   $(n-1)$--connected for all $u\in N(v)$.
\end{proof}

\begin{definition}
A \emph{squid} with \emph{body} $w$ in $G\square K_q$ is a subset of\linebreak $V(G\square K_q)$ that  is either
\begin{itemize}
\item[(i)] a \emph{subset} of
     \[ (N_G(v)\cup N_G(w)) \times\{i\} \cup \{w\}\times\{1,2,\ldots, q\} \]
for two adjacent vertices $v$ and $w$, and $1\leq i\leq q$, or
\item[(ii)] a \emph{subset} of
     \[ N_G(w)\times\{i,j\} \cup \{w\}\times\{1,2,\ldots, q\} \]
where $1\leq i<j\leq q$.
\end{itemize}
The vertices not of the form $(w,k)$ are \emph{arms}.
\end{definition}

In Figure \ref{squids} are examples of squids of both types. Note that the body of a squid is part of the data defining it, and that
two squids on the same subset of $V(G\square K_q)$ could have different bodies. For a squid $S$ in $G\square K_q$
we relax the notation and write $G\square K_q \setminus S$ for the removal of the subset of $V(G\square K_q)$
defining the squid from the graph $G\square K_q$.

\begin{figure}
\centering 
\includegraphics[width=8cm]{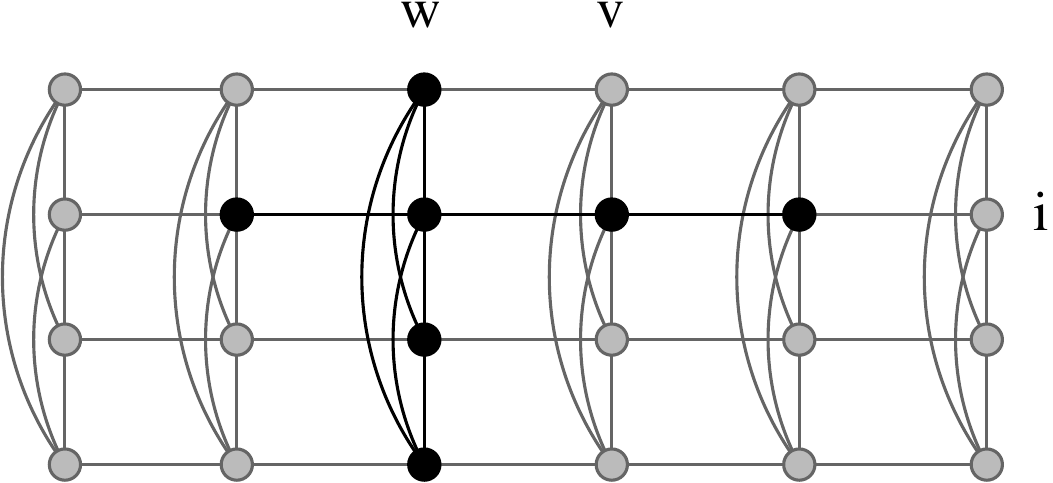}
\[ \begin{array}{c} \, \\ \, \end{array} \]
\includegraphics[width=8cm]{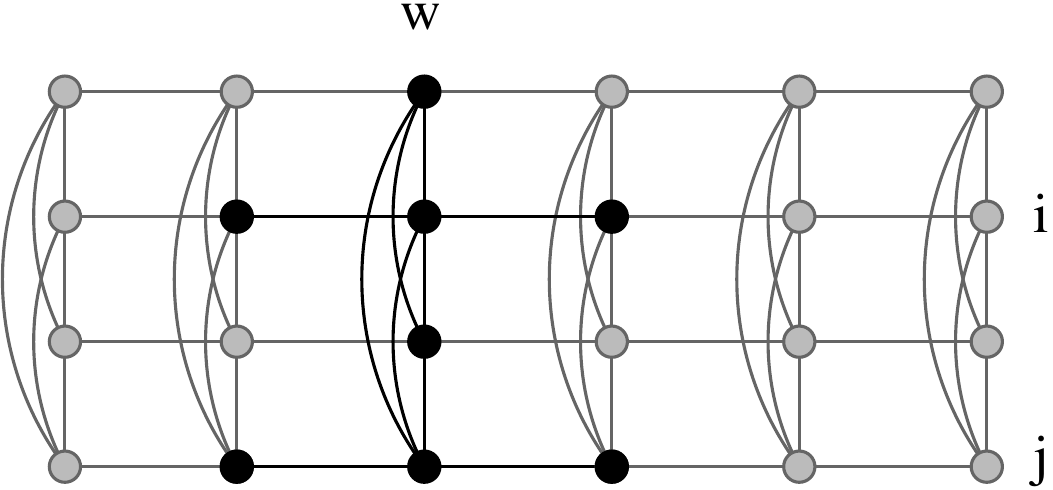}
\caption{Squids of types (i) and (ii) are drawn in black.}  \label{squids}
\end{figure} 

\begin{lemma}
\label{squidLemma}
Let $G$ be a graph with more than $m$ vertices and\linebreak $S_1, S_2, \ldots, S_{m}$ squids in $G\square K_q$ with different bodies. If $q> |N^2(v)| + 2|N(v)|$ for every vertex $v$ of $G$ then $G\square K_q \setminus \cup_{i=1}^m S_m$ is non-empty.
\end{lemma}

\begin{proof}
There is a vertex $v$ of $G$ which is not the body of a squid.
We will prove that 
$$ \{(v,1),(v,2),\ldots,(v,q)\} \setminus \cup_{i=1}^m S_m$$ 
is non-empty. Assume that out of the squids $S_1, S_2, \ldots, S_{m}$
with an arm $(v,)$ there are
\begin{itemize}
\item $a$ squids of type (i) such that its body is not adjacent with $v$,
\item $b$ squids of type (i) such that its body is adjacent with $v$,
\item $c$ squids of type (ii).
\end{itemize}
The squids of type (i) has arms at one level and squids of type (ii)
at two levels, so we have to prove that $a+b+2c<q$. Counting vertices
around $v$ in $G$ gives that $a \leq  | N^2(v) |  $
and $ b+c \leq |N(v)| $,
which using the assumption proves that
$ a+b+2c \leq a+2(b+c) =     | N^2(v) |  + 2|N(v)| < q .$

\end{proof}

\begin{theorem}\label{squidTheorem}
Let $G$ be a graph with more than $m$ vertices
and\linebreak $S_1, S_2, \ldots, S_{m}$ squids in $G\square K_q$ with different
bodies. If
$q> |N^2(v)| + 2|N(v)|$
for every vertex $v$ of $G$,
then ${\tt Ind}(G\square K_q\setminus \cup_{i=1}^m S_i)$
is $(|V(G)|-m-2)$--connected.
\end{theorem}

\begin{proof}
The proof is by induction on $|V(G)|-m$, and we start
with $m=|V(G)|-1$. We are to prove that 
$G\square K_q\setminus \cup_{i=1}^m S_i$
is non-empty, since $|V(G)|-m-2=-1$, and
${\tt Ind}(G\square K_q\setminus \cup_{i=1}^m S_i)$
is $(-1)$--connected whenever it is non-empty.
This is true by Lemma~\ref{squidLemma}. 

Now assume that $m<|V(G)|-1$. Pick a vertex $(u,i)\in G\square K_q \setminus \cup_{i=1}^m S_i$.
We can do that by Lemma~\ref{squidLemma}.
Let $H=G\square K_q \setminus \cup_{i=1}^m S_i$
and $H'=H\setminus N_G(u)\times\{i\}$.
The first step is to show that ${\tt Ind}(H')$ is $(|V(G)|-m-2)$--connected. If $(u,i)$
is isolated in $H'$ then ${\tt Ind}(H')$ is contractible and in particular
$(|V(G)|-m-2)$--connected. Otherwise $N_{H'}((u,i))$ is a complete graph and by 
Lemma~\ref{red2} it is sufficient to show that
\[{\tt Ind}(H'\setminus ( N_{H'}((u,i)) \cup N_{H'}((u,j)) )) \]
is $(|V(G)|-m-2-1)$--connected for all $j$ such that $(u,j)\in N_{H'}((u,i))$. If
we let $S'_j$ be the squid
\[ N_G(u)\times\{i,j\} \cup \{u\}\times\{1,2,\ldots, q\}, \]
then the independence complex of
$H'\setminus ( N_{H'}((u,i)) \cup N_{H'}((u,j)) ) = H \setminus S'_j$
is $(|V(G)|-m-2-1)$--connected by induction. Now we know that
${\tt Ind}(H\setminus N_G(u)\times\{i\})$ is $(|V(G)|-m-2)$--connected. 

Let
$t=|N_G(u)\times \{i\} \cap H|$, and $\{(w_1,i),(w_2,i),\ldots,(w_t,i)\}
=N_G(u)\times \{i\} \cap H$. Also let
\[H''_l =H\setminus \{(w_1,i),(w_2,i),\ldots,(w_l,i)\} \]
for $0\leq l\leq t$. Note that $H_0''=H$ and $H_t''=H'=H\setminus N_G(u)\times \{i\}$.
For $0<l\leq t$ let $S''_l$ be the type (i) squid
\[ (\{w_1,w_2,\ldots,w_l\}\cup N_G(w_l))\times \{i\} \cup
   \{w_l\}\times \{1,2,\ldots, q\}. \]
As a preparation to use Lemma~\ref{red1}, note that for
$0<l\leq t$,
$H''_l=H''_{l-1}\setminus \{(w_l,i)\}$ and $
H''_{l-1}\setminus ( N_{H''_{l-1}}((w_l,i))\cup\{(w_l,i)\})= H\setminus S_l''.$
By induction all ${\tt Ind}(H\setminus S_l'')$ are $(|V(G)|-m-2-1)$--connected.
From earlier ${\tt Ind}(H''_t)={\tt Ind}(H')$ is $(|V(G)|-m-2)$--connected. By Lemma~\ref{red1}
used with vertex $(w_t,i)$ on ${\tt Ind}(H''_{t-1})$ we get that 
${\tt Ind}(H''_{t-1})$ is $(|V(G)|-m-2)$--connected. Repeated use of Lemma~\ref{red1} for
$l=t-1,t-2,\ldots,1$ on  ${\tt Ind}(H''_{l-1})$ using the vertex $(w_l,i)$ proves
that ${\tt Ind}(H''_0)={\tt Ind}(H)$ is $(|V(G)|-m-2)$--connected.
\end{proof}

\begin{corollary}~\label{connCor}
If $q> |N^2(v)| + 2|N(v)|$ for every vertex $v$ of $G$,
then ${\tt Ind}(G\square K_q)$
is $(|V(G)|-2)$--connected.
\end{corollary}
\begin{proof}
Use no squids in Theorem~\ref{squidTheorem}.
\end{proof}

\subsection{Proof of the main Theorem}

Now we have collected enough tools to prove the theorem and corollary stated in the introduction.

\begin{theorem}\label{mainTheorem}
Let $q\geq 2$ be a prime power, $d\geq 1$, and $G$ a
graph on $(d+1)(q-1)+1$ vertices such that
\[ q> |N^2(v)| + 2|N(v)| \]
for every vertex $v$. Then $G$ is a $(d,q)$--Tverberg graph.
\end{theorem}
\begin{proof}
By Corollary~\ref{connCor}, ${\tt Ind}(G\square K_q)$ is 
$(|V(G)|-2)$--connected if
$q> |N^2(v)| + 2|N(v)|$
 for every vertex $v$ of $G$,
and by Theorem~\ref{modTver}
we are done.
\end{proof}

\begin{corollary}\label{mainCorollary}
Let $q\geq 2$ be a prime power, $d\geq 1$, and $G$ a
graph on $(d+1)(q-1)+1$ vertices such that its maximal
degree $D$ satisfy $D(D+1)<q$. Then $G$ is a $(d,q)$--Tverberg graph.
\end{corollary}

\begin{proof}
It follows directly from the equality $D(D-1)+2D=D(D+1)$, and the inequalities
$D(D-1) \geq  |N^2(v)| $ and $ D \geq |N(v)|.$
\end{proof}

We belive that the bound $D(D+1)<q$ can
be improved quite much. For which $t$ is
there a constant $k$ such that the
bound $D^t<kq$ proves that it is a 
Tverberg graph? This conjecture is 
perhaps too optimistic:

\begin{conjecture}
There is a constant $k$ such that if the
maximal degree of $G$ is less than $kq$ 
then $G$ is a $(d,q)$--Tverberg graph.
\end{conjecture} 

Some skeleton of chessboard complexes, and in
particular the chessboard complexes which are
homotopy wedges of spheres, are
shellable as proved by Ziegler \cite{Z}. His
technique can be extended to prove that
${\tt Ind}(G\square K_q)$ is shellable for certain
$G$ that are subgraphs of disjoint unions 
of small complete graphs. 

\begin{conjecture}
If $G$ is a $(d,q)$--Tverberg graph from Theorem \ref{mainTheorem}, then
${\tt Ind}(G\square K_q)$ is pure and shellable.
\end{conjecture}

\section{The affine version}

Let $\textrm{conv}(F)$ of a finite subset
$F$ of $\mathbb{R}^d$ denote the convex hull
of $F$. The Tverberg theorem states that:

\begin{theorem}[Tverberg \cite{T1}]
  For any $d\geq 1$ and $q \geq 2$, any
  set of $(d+1)(q-1)+1$ points in $\mathbb{R}^d$
  can be partitioned into $q$ disjoint subsets
  $F_1, F_2, \ldots F_q$ such that 
  $\textrm{\emph{conv}}(F_1)\cap\cdots \cap \textrm{\emph{conv}}(F_q)
   \neq \emptyset $.
\end{theorem}

Note that it was not assumed that $q$ is
a prime power. If $q$ is a prime power then
the theorem is a corollary of the topological
Tverberg theorem. Set $N=(d+1)(q-1)$. Let $v_1, v_2, \ldots v_{N+1}$
be the vertices of the geometrical realization of
a $N$--dimensional simplex and 
$p_1, p_2, \ldots p_{N+1}$ the points
in $\mathbb{R}^d$ from the statement of the Tverberg
theorem. The affine and continuous map
$$ a_1v_1+a_2v_2+\cdots+a_{N+1}v_{N+1}
   \mapsto a_1p_1+a_2p_2+\cdots+a_{N+1}p_{N+1} $$
used in the topological version proves the Tverberg
theorem.

Using the same idea we get an affine version
of Theorem~\ref{mainTheorem} and an improvement
of the Tverberg theorem in the prime power case.

\begin{corollary}\label{affineTheorem}
Let $q\geq 2$ be a prime power, $d\geq 1$, 
and $P$ be a set of $(d+1)(q-1)+1$ points
in $\mathbb{R}^d$. If $G$ is a graph with
$P$ as vertex set and
$q> |N^2(v)| + 2|N(v)|$
for all $v\in P$, then $P$ can be partitioned
into $q$ disjoint subsets
  $F_1, F_2, \ldots F_q$ such that 
  $\textrm{\emph{conv}}(F_1)\cap\cdots \cap \textrm{\emph{conv}}(F_q)
   \neq \emptyset $  and no part $F_i$ contains an edge of $G$.
\end{corollary}

There are no obvious modifications of
the known proofs of Tverbergs theorem that
would imply Corollary~\ref{affineTheorem}.
Using contemporary versions of Bertrand's postulate on the existence of primes between $q$ and $q(1+\varepsilon)$, the prime power condition can be removed, but the number of required points increases.

\section{Tverberg partitions}

According to the Tverberg theorem one can under the right conditions partition the vertices of a
simplex in a way such that the images of the faces of every part intersect. One of the main questions about the Tverberg theorem is: how many such partitions exist?
Some progress have been made \cite{H1,H2,VZ} on a very challenging conjecture by Sierksma on the number of Tverberg partitions
by using topological methods. The following theorem improves on some previous results
which correspond to the $D=1$ case.
\begin{theorem}
Let $D$ be an integer with $D(D+1)$ smaller than a prime power $q$, and $N=(d+1)(q-1)+1$ for some integer $d$.
Then the number of   $(d,q)$--Tverberg partitions are at least
$a_N/b_N$, where $a_N$ is the number of labeled $D$--regular graphs on $N$ vertices, and $b_N$ is
the maximal number of labeled $D$--regular graphs on $N$ vertices on the
same $q$--vertex coloring.
\end{theorem}
\begin{proof}
For every labeled $D$--regular graph on $N$ vertices we get at least one partition by Theorem \ref{mainTheorem}, and that partition can only be given by as many
labeled $D$--regular graphs on $N$ vertices with the
same $q$ vertex coloring, since a partition is a $q$ vertex coloring.
\end{proof}

\subsubsection*{Acknowledgements.}

The author thanks Svante Linusson, Vic Reiner, G\"unter Ziegler, and the referees for their comments.

\subsubsection*{Note added in proof.}

Conjecture 3.10 was proved in \cite{EN}.

\end{document}